\documentclass[12pt,a4paper]{article}

\usepackage{amsmath, amsthm}
\usepackage{url}

\newdimen\myMargin
\textwidth \paperwidth
\advance\textwidth -2\myMargin
\textheight \paperheight
\advance\textheight -2\myMargin
\advance\oddsidemargin \myMargin
\advance\evensidemargin \myMargin
\advance\topmargin \myMargin
\advance\textheight -12 true mm

\def\sgn{\mathop{\rm sgn}\nolimits}
\def\M{\mathcal{M}}
\def\S{\mathcal{S}}
\def\Sym#1{\mathrm{Sym}(#1)}
\newtheorem{Thm}{Theorem}
\newtheorem*{Thm*}{Theorem}
\newtheorem{Lem}{Lemma}
\newtheorem*{Lem*}{Lemma}
\newtheorem{Def}{Definition}
\newtheorem{Cor}{Corollary}

\begin{document}

\title{On the grasshopper problem with signed jumps}
\author{G\'eza K\'os}
\date{}
\maketitle


\begin{abstract}
  The 6th problem of the 50th International Mathematical Olympiad
  (IMO), held in Germany, 2009, was the following.  \textit{Let
    $a_1,a_2,\dots,a_n$ be distinct positive integers and let $\M$ be
    a set of $n-1$ positive integers not containing
    $s=a_1+a_2+\dots+a_n$. A grasshopper is to jump along the real
    axis, starting at the point $0$ and making $n$ jumps to the right
    with lengths $a_1,a_2,\dots,a_n$ in some order. Prove that the
    order can be chosen in such a way that the grasshopper never lands
    on any point in $\M$.}  In this paper we consider a variant of the
  IMO problem when the numbers $a_1,a_2,\dots,a_n$ can be negative as
  well. We find the sharp minimum of the cardinality of the set $\M$
  which blocks the grasshopper, in terms of $n$. In contrast with the
  Olympiad problem where the known solutions are purely combinatorial,
  for the solution of the modified problem we use the polynomial
  method.
\end{abstract}


\section{Introduction}

\subsection{The original Olympiad problem}
\label{sect:imo}

The 6th problem of the 50th International Mathematical Olympiad (IMO),
held in Germany, 2009, was the following.

\begin{quote}
  Let $a_1,a_2,\dots,a_n$ be distinct positive integers and let $\M$
  be a set of $n-1$ positive integers not containing
  $s=a_1+a_2+\dots+a_n$. A grasshopper is to jump along the real axis,
  starting at the point $0$ and making $n$ jumps to the right with
  lengths $a_1,a_2,\dots,a_n$ in some order. Prove that the order can
  be chosen in such a way that the grasshopper never lands on any
  point in $\M$.
\end{quote}

For $n\ge2$ the statement of the problem is sharp. For arbitrary
positive numbers $a_1,\dots,a_n$ it is easy to find a ``mine field''
$\M$ of $n$ mines which makes the grasshopper's job impossible. Such
sets are, for example, $\M_1=\{a_1,a_2,\dots,a_n\}$ and
$\M_2=\{s-a_1,\dots,s-a_n\}$. In special cases more examples can be
found; for example if $(a_1,a_2,\dots,a_n)=(1,2,\dots,n)$ then any $n$
consecutive integers between $0$ and $s$ block the grasshopper.

\medskip

The problem has been discussed in many on-line forums, as much by
communities of students as by senior mathematicians; see, for example,
the Art of Problem Solving / Mathlinks forum~\cite{AoPS} or Terence
Tao's Mini-polymath project~\cite{TaoForum}.  Up to now, all known
solutions to the Olympiad problem are elementary and inductive.

\medskip


\subsection{Attempts to use the polynomial method}

Some students at the test, who were familiar with the polynomial
method, tried to apply Noga Alon's combinatorial Nullstellensatz (see
Lemmas~1.1 and~1.2 in~\cite{CN}). In this technique the problem is
encoded via a polynomial whose nonzeros are solutions in some given
domain.  A powerful tool to prove the existence of a point where the
polynomial does not vanish is the so-called \emph{combinatorial
  Nullstellensatz}.

\begin{Lem}[Nonvanishing criterion of the combinatorial
  Nullstellensatz]
  \label{thm:CN}
  Let $\S_1,\dots,\S_n$ be nonempty subsets of a field $F$, and let
  $t_1,\dots,t_n$ be nonnegative integers such that $t_i<|\S_i|$ for
  $i=1,2\,\dots,n$. Let $P(x_1,\dots,x_n)$ be a polynomial over $F$
  with total degree $t_1+\dots+t_n$, and suppose that the coefficient
  of $x_1^{t_1}x_2^{t_2}\cdots x_n^{t_n}$ in $P(x_1,\dots,x_n)$ is
  nonzero. Then there exist elements $s_1\in\S_1$, \dots, $s_n\in\S_n$
  for which $P(s_1,\dots,s_n)\ne0$.
\end{Lem}

If we want to use the combinatorial Nullstellensatz to solve the
grasshopper problem, it seems promising to choose
$\S_1=\dots=\S_n=\{a_1,\dots,a_n\}$ and the polynomial
\begin{equation}
  \label{eq:vp}
  P(x_1,\dots,x_n) =
  V(x_1,\dots,x_n) \cdot \prod_{\ell=1}^{n-1} \prod_{m\in \M} \big(
  (x_1+\dots+x_\ell) - m \big)
\end{equation}
where
$$
V(x_1,\dots,x_n) = \prod_{1\le i<j\le n} (x_j-x_i) =
\sum_{\pi\in\Sym{n}} \sgn(\pi) \cdot x_{\pi(2)} x_{\pi(3)}^2 \cdots x_{\pi(n)}^{n-1}
$$
is the so-called Vandermonde polynomial (see, for
example,~\cite[pp. 346--347]{AignerBook}), the symbol $\Sym{n}$
denotes the group of the permutations of the sequence $(1,2,\dots,n)$,
and $\sgn(\pi)$ denotes the sign of the permutation~$\pi$.

If there exist some values $s_1,\dots,s_n\in\{a_1,\dots,a_n\}$ such
that $P(s_1,\dots,s_n)\ne0$, then it follows that the grasshopper can
choose the jumps $s_1,\dots,s_n$. The value of the Vandermonde
polynomial is nonzero if and only if the jumps $s_1,\dots,s_n$ are
distinct, i.e., $(s_1,\dots,s_n)$ is a permutation of
$(a_1,\dots,a_n)$, and a nonzero value of the double product on the
right-hand side of \eqref{eq:vp} indicates that this order of jumps
provides a safe route to the point~$s$.
(Similar polynomials appear in various applications of the
combinatorial Nullstellensatz; see, for example,~\cite{Snevily1} or
~\cite{Snevily2}.)

The first difficulty we can recognize is that the degree is too high.
To apply the combinatorial Nullstellensatz to the sets
$\S_1=\dots=\S_n=\{a_1,\dots,a_n\}$, we need $\deg P=t_1+\dots+t_n\le
n(n-1)$, while the degree of our polynomial $P(x_1,\dots,x_n)$ is
$\frac12(3n-2)(n-1)$. But there is another, less obtrusive deficiency:
this approach does not use the condition that the numbers
$a_1,\dots,a_n$ are positive. We may ask: is this condition really
important? What happens if this condition is omitted?  In this paper
we answer these questions.


\subsection{Modified problem with signed jumps}

We consider a modified problem in which the numbers $a_1,\dots,a_n$
are allowed to be negative as well as positive.  A positive value
means that the grasshopper jumps to the right, while a negative value
means a jump to the left. We determine the minimal size, in terms of
$n$, of a set $\M$ that blocks the grasshopper.

Optionally we may allow or prohibit the value $0$ among the numbers
$a_1,\dots,a_n$; such a case means that the grasshopper ``hops'' to
the same position. This option does not make a significant difference,
but affects the maximal size of the set~$\M$.

\medskip

In order to find a sharp answer, let us collect a few simple cases
when the grasshopper gets blocked. Table~\ref{tab:examples} shows
three such examples. If $n$ is odd, the size of the set $\M$ is
different if hops are allowed or prohibited.

\begin{table}[th!]
  \begin{center}
    \begin{tabular}{|l|l|l|}
      \hline Case & $\{a_1,a_2,\dots,a_n\}$ & $\M$ \\ \hline
      $n=2k$   & $\{-k+1,\dots,k+1\}\setminus\{0\}$ & $\{1,\dots,k+1\}$ \\
      \hline
      $n=2k+1$; ``hops'' allowed &
      $\{-k+1, \dots,k+1\}$ & $\{1,\dots,k+1\}$ \\ \hline
      $n=2k+1$; ``hops'' prohibited &
      $\{-k+1,\dots,k+2\}\setminus\{0\}$ & $\{1,\dots,k+2\}$ \\ \hline
    \end{tabular}
  \end{center}
  \caption{Simple cases when the grasshopper is blocked.
    \label{tab:examples}}
\end{table}

In Section~\ref{Sect:main} we show that these examples are minimal: if
the size of the set $\M$ is smaller than in the examples above then
the desired order of jumps exists, as stated in the next theorem,
which is our main result in this paper.
\begin{Thm}
  \label{thm:tn}
  Suppose that $a_1,a_2,\dots,a_n$ are distinct integers and $\M$ is a
  set of integers with
  $\displaystyle|\M|\le\left\lfloor\frac{n}{2}\right\rfloor$. Then
  there exists a permutation $(b_1,\dots,b_n)$ of the sequence
  $(a_1,\dots,a_n)$ such that none of the sums $b_1$, $b_1+b_2$,
  \dots, $b_1+b_2+\dots+b_{n-1}$ is an element of~$\M$.

  If the numbers $a_1,a_2,\dots,a_n$ are all nonzero, then the same
  holds for $\displaystyle|\M|\le\left\lfloor\frac{n+1}{2}\right\rfloor$ as well.
\end{Thm}

Notice that in the modified problem the set $\M$ contains fewer mines
than in the statement in the Olympiad problem. This resolves the
difficulties about the degree of the polynomial, allowing us to use
the polynomial method.


\bigskip

\section{Solution to the modified problem}
\label{Sect:main}

In Sections~\ref{sect:polys} and~\ref{sect:ck} we prove
Theorem~\ref{thm:tn} in the particular case when $n$ is even. For the
proof we use the polynomial method. The approach we follow does not
apply the combinatorial Nullstellensatz directly, but there is a close
connection to it; the connection between the two methods is discussed
in Section~\ref{sect:altCN}. Finally, Theorem~\ref{thm:tn} is proved
for odd values of~$n$ in Section~\ref{sect:oddN}.


\subsection{Setting up the polynomials when $n$ is even}
\label{sect:polys}

Throughout Sections~\ref{sect:polys}--\ref{sect:altCN} we assume that
$n=2k$. Since we can add arbitrary extra elements to the set $\M$, we
can also assume $|\M|=k$ without loss of generality.

\medskip

Define the polynomial
\begin{equation}
  \label{eq:pol}
  Q(x_1,\dots,x_{2k}) = \sum_{\pi\in\Sym{2k}} \sgn(\pi)
  \prod_{\ell=1}^{2k-1} \prod_{m\in \M}
  \big((x_{\pi(1)}+x_{\pi(2)}+\dots+x_{\pi(\ell)})-m\big)
\end{equation}
and consider the number $Q(a_1,\dots,a_{2k})$.  If
$Q(a_1,\dots,a_{2k})\ne0$ then there is a permutation
$\pi\in\Sym{2k}$ for which
$$
\prod_{\ell=1}^{2k-1} \prod_{m\in \M}
\big((a_{\pi(1)}+a_{\pi(2)}+\dots+a_{\pi(\ell)})-m\big) \ne 0.
$$
This relation holds if and only if the numbers
$$
a_{\pi(1)},
\quad
a_{\pi(1)}+a_{\pi(2)},
\quad \dots, \quad
a_{\pi(1)}+a_{\pi(2)}+\dots+a_{\pi(2k-1)}
$$
all differ from the elements of~$\M$. Hence, it is sufficient to prove
$Q(a_1,\dots,a_{2k})\ne0$.

\medskip

Since $Q$ is an alternating polynomial, it is a multiple of the
Vandermonde polynomial $V(x_1,\dots,x_{2k})$ (we refer again
to~\cite[pp. 346--347]{AignerBook}). The degree of $Q$ is at most
$k(2k-1)$, which matches the degree of $V$, so
\begin{equation}
  \label{eq:nomine}
  Q(x_1,x_2,\dots,x_{2k}) =
  (-1)^k c_k \cdot V(x_1,x_2,\dots,x_{2k})
\end{equation}
for some constant $c_k$. (The sign $(-1)^k$ on the right-hand side is
inserted for convenience, to make the coefficient of
$x_1^{2k-1}x_2^{2k-2}\cdots x_{2k-1}$ positive.) Substituting
$a_1,\dots,a_{2k}$, we obtain
$$
Q(a_1,a_2,\dots,a_{2k}) =
(-1)^k c_k \cdot V(a_1,a_2,\dots,a_{2k}).
$$
Since $a_1,\dots,a_{2k}$ are distinct,
$V(a_1,a_2,\dots,a_{2k})\ne0$. So it is left to prove $c_{k}\ne0$.

\medskip

The equation \eqref{eq:nomine} shows that the polynomial $Q$ and the
constant $c_k$ do not depend on the set $\M$; the mines are
canceled out. Therefore we get the same polynomial if all mines are
replaced by $0$:
$$
Q(x_1,\dots,x_{2k}) =
\sum_{\pi\in\Sym{2k}} \sgn(\pi) \prod_{\ell=1}^{2k-1}
\big(x_{\pi(1)}+\cdots+x_{\pi(\ell)}\big)^k =
(-1)^k c_k \cdot V(x_1,\dots,x_{2k}).
$$
Hence, the constant $c_k$ depends only on the value of~$k$.

\medskip

Due to the complexity of the polynomial $Q$, it is not straightforward to
access the constant $c_k$.  The first values are listed in Table~\ref{tab:ck}.
\begin{table}[th!]
  \begin{center}
    \begin{tabular}{ll}
      $c_1=1$                     & $c_6 = 7\,886\,133\,184\,567\,796\,056\,800$ \\
      $c_2=2$                     & $c_7 \approx 8.587\cdot10^{34}$    \\
      $c_3=90$                    & $c_8 \approx 4.594\cdot10^{51}$    \\
      $c_4=586\,656$              & $c_9 \approx 2.060\cdot10^{72}$    \\
      $c_5=1\,915\,103\,977\,500$ & $c_{10} \approx 1.237\cdot10^{97}$ \\
    \end{tabular}
  \end{center}
  \caption{Values of $c_k$ for small~$k$.
    \label{tab:ck}}
\end{table}


\subsection{Proof of $c_k\ne0$}
\label{sect:ck}

We prove a more general statement which contains $c_k>0$ as a special
case.


\begin{Def}
  For every $n\ge1$ and pair $u,v\ge0$ of integers, define the polynomial
  \begin{equation}
    \label{eq:L1}
    Q^{(n,u,v)}(x_1,\dots,x_n) =
    \sum_{\pi\in\Sym{n}} \sgn(\pi) \left( \prod_{\ell=1}^{n-1}
    \big(x_{\pi(1)}+\dots+x_{\pi(\ell)}\big)^u \right)
    (x_1+\dots+x_n)^v.
  \end{equation}
  In the case $n=1$, when the product
  $\prod_{\ell=1}^{n-1}(\dots)^u$ is empty, let
  $Q^{(1,u,v)}(x_1) = x_1^v$.

  For every sequence $d_1\ge\dots\ge d_n\ge0$ of integers, denote by
  $\alpha^{(n,u,v)}_{d_1,\dots,d_n}$ the coefficient of the monomial
  $x_1^{d_1} \cdots x_n^{d_n}$ in $Q^{(n,u,v)}$.  For convenience, define
  $\alpha^{(n,u,v)}_{d_1,\dots,d_n}=0$ for $d_n<0$ as well.
\end{Def}

Since the polynomial $Q^{(n,u,v)}$ is alternating, we have
$\alpha^{(n,u,v)}_{d_1,\dots,d_n}=0$ whenever $d_i=d_{i+1}$ for some
index $1\le i<n$, and we can write
\begin{equation}
  \label{eq:L2}
  Q^{(n,u,v)}(x_1,\dots,x_n) =
  \sum_{d_1>\dots>d_n\ge0} \alpha^{(n,u,v)}_{d_1,\dots,d_n}
  \sum_{\pi\in\Sym{n}} \sgn(\pi) \,
  x_{\pi(1)}^{d_1} x_{\pi(2)}^{d_2} \cdots x_{\pi(n)}^{d_n}.
\end{equation}


\begin{Lem}
  \label{thm:lem1}
  If $n\ge2$ and $d_1>\dots>d_n\ge 0$ then
  $$
  \alpha^{(n,u,0)}_{d_1,\dots,d_n} =
  \begin{cases}
    \alpha^{(n-1,u,u)}_{d_1,\dots,d_{n-1}} & \text{if $d_n=0$;} \\
    0 & \text{if $d_n>0$.}
  \end{cases}
  $$
\end{Lem}

\begin{proof}
  Consider the polynomial
  $$
  Q^{(n,u,0)}(x_1,\dots,x_n)
  = \sum_{\pi\in\Sym{n}} \sgn(\pi) \prod_{\ell=1}^{n-1}
  \big(x_{\pi(1)}+\dots+x_{\pi(\ell)}\big)^u.
  $$
  The product $\prod_{\ell=1}^{n-1}
  \big(x_{\pi(1)}+\dots+x_{\pi(\ell)}\big)^u$ does not contain the
  variable $x_{\pi(n)}$. Therefore the term $x_{1}^{d_1}\cdots
  x_{n-1}^{d_{n-1}}x_{n}^{d_n}$ can occur only for $d_n=0$ and
  $\pi(n)=n$.  Summing over such permutations,
  $$
  \sum_{\pi\in\Sym{n},\,\pi(n)=n} \sgn(\pi) \prod_{\ell=1}^{n-1}
  \big(x_{\pi(1)}+\dots+x_{\pi(\ell)}\big)^u =
  Q^{(n-1,u,u)}(x_1,\dots,x_{n-1}).
  $$
  Hence, the coefficients of $x_1^{d_1}\cdots x_{n-1}^{d_{n-1}}$ in the
  polynomials $Q^{(n,u,0)}$ and $Q^{(n-1,u,u)}$ are the same.
\end{proof}


\begin{Lem}
  \label{thm:lem2}
  If $v\ge1$ and $d_1>\dots>d_n\ge0$ then
  $$
  \alpha^{(n,u,v)}_{d_1,\dots,d_n} = \sum_{i=1}^n
  \alpha^{(n,u,v-1)}_{d_1,\dots,d_{i-1},d_i-1,d_{i+1},\dots,d_n}.
  $$
\end{Lem}

\begin{proof}
  By the definition of the polynomial $Q^{(n,u,v)}$, we have
  \begin{eqnarray*}
    Q^{(n,u,v)}(x_1,\dots,x_n) &=&
    (x_1+\dots+x_n) \cdot Q^{(n,u,v-1)}(x_1,\dots,x_n) \\
    &=& \sum_{i=1}^n x_i \cdot Q^{(n,u,v-1)}(x_1,\dots,x_n) .
  \end{eqnarray*}
  On the left-hand side, the coefficient of $x_1^{d_1}\cdots x_{n}^{d_{n}}$ is
  $\alpha^{(n,u,v)}_{d_1,\dots,d_n}$. On the right-hand side, the coefficient
  of $x_1^{d_1}\cdots x_{n}^{d_{n}}$ in $x_i \cdot Q^{(n,u,v-1)}$ is
  $\alpha^{(n,u,v-1)}_{d_1,\dots,d_{i-1},d_i-1,d_{i+1},\dots,d_n}$.
\end{proof}


\begin{Lem}
  \label{thm:lem3}
  If $n\ge1$, $u,v\ge0$, and $d_1>\dots>d_n\ge0$ then
  $$ \alpha^{(n,u,v)}_{d_1,\dots,d_n}\ge0 . $$
\end{Lem}

\begin{proof}
  Apply induction on the value of $n(u+1)+v$.

  If $n=1$ then we have $Q^{(1,u,v)}(x_1) = x_1^v$, and the only
  nonzero coefficient is positive.

  Now suppose that $n\ge2$, and the induction hypothesis holds for all
  cases when $n(u+1)+v$ is smaller. Then the induction step is
  provided by Lemma~\ref{thm:lem1} for $v=0$, and by
  Lemma~\ref{thm:lem2} for $v\ge1$.
\end{proof}


\begin{Lem}
  \label{thm:lem4}
  Suppose that $n\ge1$, $u\ge n/2$, $v\ge0$, and $d_1>\dots>d_n\ge0$
  are integers such that $d_1+\dots+d_n=(n-1)u+v$ and $d_n\le v$.
  Then
  $$ \alpha^{(n,u,v)}_{d_1,\dots,d_n}>0. $$
\end{Lem}

\begin{proof}
  Apply induction on the value of $n(u+1)+v$.

  In case $n=1$ the condition $d_1+\dots+d_n=(n-1)u+v$ reduces to
  $d_1=v$. Then we have $Q^{(1,u,v)}(x_1)=x_1^v$ and thus
  $\alpha^{(n,u,v)}_{d_1}=1>0$.

  \medskip

  Now suppose that $n\ge2$ and the induction hypothesis holds for all
  cases when $n(u+1)+v$ is smaller.  We consider two cases, depending on
  the value of~$v$.

  \medskip\noindent\textit{Case 1: $v=0$.}
  From the condition $d_n\le v$ we get $d_n=0$. Then, by
  Lemma~\ref{thm:lem1},
  $$
  \alpha^{(n,u,0)}_{d_1,\dots,d_{n-1},0} = \alpha^{(n-1,u,u)}_{d_1,\dots,d_{n-1}}.
  $$

  Apply the induction hypothesis to $(n',u',v')=(n-1,u,u)$ and
  $(d'_1,\dots,d'_{n-1})=(d_1,\dots,d_{n-1})$.  The condition $u'\ge
  n'/2$ is satisfied, because $u'=u\ge n/2 > n'/2$.
  Since $v=d_n=0$, we have
  $$
  d'_1+d'_2+\dots+d'_{n'} = d_1+d_2+\dots+d_{n-1}+d_n
  = (n-1)u = (n'-1)u' + v'.
  $$

  \noindent
  Finally,
  $$
  d'_{n'} = d_{n-1} \le \frac{d_1+d_2+\dots+d_{n-1}}{n-1} = u = v'.
  $$

  Hence, the numbers $n',u',v'$, and $d'_1,\dots,d'_{n-1}$ satisfy the
  conditions of the lemma. By the induction hypothesis and
  Lemma~\ref{thm:lem1} we can conclude
  $$
  \alpha^{(n,u,v)}_{d_1,\dots,d_{n-1},d_n} =
  \alpha^{(n,u,0)}_{d_1,\dots,d_{n-1},0} =
  \alpha^{(n-1,u,u)}_{d_1,\dots,d_{n-1}}>0.
  $$


  \medskip\noindent\textit{Case 2: $v>0$.}
  By Lemma~\ref{thm:lem2}, we have
  $$
  \alpha^{(n,u,v)}_{d_1,\dots,d_n} =
  \sum_{i=1}^n \alpha^{(n,u,v-1)}_{d_1,\dots,d_{i-1},d_i-1,d_{i+1},\dots,d_n}.
  $$
  By Lemma~\ref{thm:lem3}, all terms are nonnegative on the right-hand
  side. We show that there is at least one positive term among them.

  Since
  $$
  d_1+\dots+d_n = (n-1)u+v \ge (n-1)\cdot\frac{n}2+1 > (n-1)+(n-2)+\dots+1+0,
  $$
  there exists an index $i$, $1\le i\le n$, for which $d_i>n-i$. Let
  $i_0$ be the largest such index. If $i_0<n$ then
  $d_{i_0}-1>d_{i_0+1}$. Otherwise, if $i_0=n$, we have $d_n-1\ge0$.

  Now apply the induction hypothesis to $(n',u',v')=(n,u,v-1)$ and
  $(d'_1,\dots,d'_n)=(d_1,\dots,d_{i_0-1},d_{i_0}-1,d_{i_0+1},\dots,d_n)$.
  By the choice of $i_0$ we have $d'_1>\dots>d'_n\ge0$.  The condition
  $u'\ge n'/2$ obviously satisfied.  Finally, $d'_n\le v$ holds,
  too, due to $d'_n=\max(d_n-1,0) \le v-1$.

  Hence, the induction hypothesis can be applied and we conclude
  $$ \alpha^{(n,u,v)}_{d_1,\dots,d_n} \ge
  \alpha^{(n,u,v-1)}_{d_1,\dots,d_{i_0-1},d_{i_0}-1,d_{i_0+1},\dots,d_n}
  > 0.
  $$
\end{proof}


\begin{Cor}
  \label{thm:ckpositive}
  $c_k>0$ for every positive integer $k$.
\end{Cor}

\begin{proof}
  Apply Lemma~\ref{thm:lem4} with $n=2k$, $u=k$, $v=0$, and
  $(d_1,\dots,d_{2k})=(2k-1,2k-2,\dots,1,0)$. The conditions of the
  lemma are satisfied, so
  $$ c_k = \alpha^{(2k,k,0)}_{2k-1,2k-2,\dots,1,0} > 0. $$
\end{proof}

Corollary~\ref{thm:ckpositive} completes the proof of
Theorem~\ref{thm:tn} when $n$ is even.


\subsection{An alternative approach using the combinatorial
  Nullstellensatz}
\label{sect:altCN}

It is natural to ask whether the combinatorial Nullstellensatz is
applicable to solve the modified problem, as was mentioned in the
introduction. As in Section~\ref{sect:imo}, let
$\S_1=\dots=\S_{2k}=\{a_1,\dots,a_{2k}\}$, $t_1=\dots=t_{2k}=2k-1$, and
consider the polynomial
$$
P(x_1,\dots,x_{2k}) =
V(x_1,x_2,\dots,x_{2k}) \cdot
\prod_{\ell=1}^{2k-1} \prod_{m\in \M}
\big((x_1+x_2+\dots+x_\ell)-m\big).
$$
The total degree of $P$ is $\binom{2k}{2}+(2k-1)k=2k(2k-1)$. To apply
the nonvanishing lemma, it is sufficient to prove that the coefficient
of the monomial $x_1^{2k-1}x_2^{2k-1}\cdots x_{2k}^{2k-1}$ in $P$ is
nonzero.  We show this coefficient is exactly $c_k$.

\bigskip


Compare the coefficient of $x_1^{2k-1}x_2^{2k-1}\cdots
x_{2k}^{2k-1}$ in  the polynomials
\begin{eqnarray}
  P(x_1,\dots,x_{2k}) &=&
  \sum_{\pi\in\Sym{2k}} \sgn(\pi) x_{\pi(2)}x_{\pi(3)}^2\cdots x_{\pi(2k)}^{2k-1}
  \prod_{\ell=1}^{2k-1} \prod_{m\in \M} \big((x_1+\cdots+x_\ell)-m)\big),
  \nonumber \\
  P_1(x_1,\dots,x_{2k}) &=&
  \sum_{\pi\in\Sym{2k}} \sgn(\pi) x_{\pi(2)}x_{\pi(3)}^2\cdots x_{\pi(2k)}^{2k-1}
  \prod_{\ell=1}^{2k-1} (x_1+\cdots+x_\ell)^k,
  \nonumber
\end{eqnarray}
and
\begin{eqnarray}
  P_2(x_1,\dots,x_{2k}) &=&
  \sum_{\pi\in\Sym{2k}} \sgn(\pi) x_2x_3^2\cdots x_{2k}^{2k-1}
  \prod_{\ell=1}^{2k-1}
  (x_{\pi^{-1}(1)}+\cdots+x_{\pi^{-1}(\ell)})^k
  \nonumber \\
  &=& x_2x_3^2\cdots x_{2k}^{2k-1} \cdot Q(x_1,\dots,x_{2k})
  \nonumber \\
  &=& c_k \cdot x_2x_3^2\cdots x_{2k}^{2k-1} \cdot (-1)^k V(x_1,\dots,x_{2k})
  \nonumber \\
  &=& c_k \cdot x_2x_3^2\cdots x_{2k}^{2k-1} \cdot V(x_{2k},\dots,x_1).
  \label{eq:q2}
\end{eqnarray}
The difference between $P$ and $P_1$ is only in the constants $m$; the
maximal degree terms are the same.

For a fixed $\pi\in\Sym{2k}$, the polynomials
$$
x_{\pi(2)}x_{\pi(3)}^2\cdots x_{\pi(2k)}^{2k-1}
\prod_{\ell=1}^{2k-1} (x_1+\cdots+x_\ell)^k \\
$$
and
$$
x_2x_3^2\cdots x_{2k}^{2k-1}
\prod_{\ell=1}^{2k-1} (x_{\pi^{-1}(1)}+\cdots+x_{\pi^{-1}(\ell)})^k
$$
differ only in the order of the variables. Hence, the coefficients of
the monomial $x_1^{2k-1}x_2^{2k-1}\cdots x_{2k}^{2k-1}$ are the same
in $P$, $P_1$, and $P_2$.  From the last line of \eqref{eq:q2} it can
be seen that this coefficient is~$c_k$.

\medskip

Hence, the method of alternating sums (Section \ref{sect:polys}) and
applying the combinatorial Nullstellensatz as above are closely
related and lead to the same difficulty, i.e., to proving~$c_k\ne0$.


\subsection{The case of odd $n$}
\label{sect:oddN}

Now we finish the proof of Theorem~\ref{thm:tn}.  Let $n=2k+1$ be an
odd number. The case $n=1$ is trivial, so assume $k\ge1$.

The proof works by inserting the value $0$ into, or removing it from,
the list $a_1,\ldots,a_n$ of jumps and then applying
Theorem~\ref{thm:tn} in the even case which is already proved.

\medskip

\noindent\textit{Case 1: The value $0$ appears among
  $a_1,\dots,a_{2k+1}$.}
Without loss of generality we can assume $a_{2k+1}=0$.  Suppose
${|\M|\le\left\lfloor n/2 \right\rfloor=k}$, and apply
Theorem~\ref{thm:tn} to the numbers $a_1,\ldots,a_{2k}$ and the
set~$\M$. The theorem provides a permutation $(b_1,\dots,b_{2k})$ of
$(a_1,\dots,a_{2k})$ such that $b_1$, $b_1+b_2$, \dots,
$b_1+\dots+b_{2k-1}$ are not elements of $\M$. Insert the value $0$
somewhere in the middle, say between the first and second
positions. Then $(b_1,0,b_2,\dots,b_{2k})$ is a permutation of
$(a_1,\dots,a_{2k},0)$, with all the required properties.

\medskip

\noindent\textit{Case 2: The numbers $a_1,\dots,a_{2k+1}$ are all
  nonzero.}
Suppose $|\M|\le\left\lfloor(n+1)/2\right\rfloor=k+1$, and apply
Theorem~\ref{thm:tn} to the numbers $a_1,\dots,a_{2k+1},a_{2k+2}=0$
and the set $\M$. By the theorem, there is a permutation
$(b_1,\dots,b_{2k+2})$ of $(a_1,\dots,a_{2k+1},0)$ such that none of
$b_1$, \dots, $b_1+\dots+b_{2k+1}$ is an element of $\M$.  Deleting
the value $0$ from the sequence $(b_1,\dots,b_{2k+2})$, we obtain the
permutation we want. \qed


\bigskip

\section{Closing remarks}

The high degree of $P$ prevented the application of Alon's
combinatorial Nullstellensatz from solving the original Olympiad
problem. We have demonstrated that if the constraint that the jumps
are positive is removed, it allows us to use the polynomial
method. This also shows that the sign condition was the real reason
behind the degree being too high.


\subsection{Extension to finite fields}

Finite fields and commutative groups need further consideration. The
examples in Table~\ref{tab:examples} are also valid in cyclic
(additive) groups, in particular in prime fields. However, our proof
does not work directly for prime fields, because the constants $c_k$
have huge prime divisors, as demonstrated below.
\begin{eqnarray*}
  c_3 &=& 2 \cdot 3^2 \cdot 5 \\
  c_4 &=& 2^5 \cdot 3^3 \cdot 7 \cdot 97 \\
  c_5 &=& 2^2 \cdot 3 \cdot 5^4 \cdot 7 \cdot 79 \cdot 103 \cdot 4\,483 \\
  c_6 &=& 2^5 \cdot 3^6 \cdot 5^2 \cdot 11 \cdot 23 \cdot 223 \cdot 239 \cdot
  1\,002\,820\,739.
\end{eqnarray*}
So, for finite fields and groups, the problem is not yet closed.


\subsection{Alternating sums vs. combinatorial Nullstellensatz}

When we need a permutation of a finite sequence of numbers with some
required property, the method applied in Section~\ref{sect:polys} can
be a replacement for the combinatorial Nullstellensatz.

\medskip

Let $a_1,\dots,a_n$ be distinct elements of a field $F$, and let $R\in
F[x_1,\dots,x_n]$ be a polynomial with degree $\binom{n}{2}$. Suppose that we
need a permutation $(b_1,\dots,b_n)$ of $(a_1,\dots,a_n)$ such that
$R(b_1,\dots,b_n)\ne0$. For such problems, it has become standard to apply the
combinatorial Nullstellensatz to the polynomial $V\cdot R$. The method we used
is different: we considered the value
\begin{equation}
  \label{eq:modpol1}
  \sum_{\pi\in\Sym{n}} \sgn(\pi) R(a_{\pi(1)},\dots,a_{\pi(n)})
  = c\cdot V(a_1,a_2,\dots,a_n)
\end{equation}
where the constant $c$ is (up to the sign) the same as the coefficient
of $x_1^{n-1}\cdots x_n^{n-1}$ in the polynomial
$V(x_1,\dots,x_n)\cdot R(x_1,\dots,x_n)$, as was emphasized in
Section~\ref{sect:altCN}.

\medskip

Note, however, that it is easy to extend the alternating sum approach
of~\eqref{eq:modpol1} to a more general setting. By inserting an
arbitrary auxiliary polynomial $A\in F[x_1,\dots,x_n]$, we may
consider the number
\begin{equation}
  \label{eq:modpol2}
  \sum_{\pi\in\Sym{n}} \sgn(\pi) \cdot
  R(a_{\pi(1)},\dots,a_{\pi(n)}) \cdot
  A(a_{\pi(1)},\dots,a_{\pi(n)})
\end{equation}
as well. It may happen that the sum in~\eqref{eq:modpol2} is not $0$
even when the sum of~\eqref{eq:modpol1} vanishes.


\subsection*{Appendix: A solution for the Olympiad problem}

For the sake of completeness, we outline a solution for the original
Olympiad problem in which all jumps were positive. We prove the
following statement.

\begin{Thm*}
  Suppose that $a_1,a_2,\dots,a_n$ are distinct positive real numbers
  and let $\M$ be a set of at most $n-1$ real numbers. Then there
  exists a permutation $(i_1,\dots,i_n)$ of $(1,2,\dots,n)$ such that
  none of the numbers $a_{i_1}+a_{i_2}+\dots+a_{i_k}$ ($1\le k\le
  n-1$) is an element of~$\M$.
\end{Thm*}

\begin{proof}
  We employ induction on $n$. In the case of $n=1$ the statement is
  trivial since the set $\M$ is empty; moreover there is no $k$ with
  $1\le k\le n-1$.

  Suppose that $n\ge2$ and the statement is true for all smaller
  values. Without loss of generality, we can assume that $a_n$ is the
  greatest among $a_1,a_2,\dots,a_n$. Since we can insert extra
  elements into $\M$, we can also assume that $\M$ has $n-1$ elements
  which are $m_1<\dots<m_{n-1}$.

  If $a_1+\dots+a_{n-1}<m_1$ then we can choose $i_n=n$ and any order
  of the indices $1,2,\dots,n-1$. So we can assume $m_1\le
  a_1+\dots+a_{n-1}$ as well.

  \medskip

  We consider two cases.

  \medskip

  \textit{Case 1: $a_n\in\M$.}
  Let $\ell$ be the index for which $a_n=m_\ell$, and let
  $\M'=\{m_1,\dots,m_{\ell-1}\}\cup\{m_{\ell+1}-a_n,\dots,m_{n-1}-a_n\}$.
  By applying the induction hypothesis to the jumps
  $a_1,\dots,a_{n-1}$ and the set $\M'$, we get a permutation
  $(j_1,\dots,j_{n-1})$ of the indices $(1,2,\dots,n-1)$ such that
  $a_{j_1}+a_{j_2}+\dots+a_{j_k}\not\in\M'$ for any $k$ such that
  $1\le k\le n-2$.

  Choose $(i_1,\dots,i_n)=(j_1,n,j_2,j_3,\dots,j_{n-1})$. We claim
  that this permutation of $(1,2,\dots,n)$ satisfies the required
  property.

  By the choice of $i_1$, we have
  $a_{i_1}=a_{j_1}\not\in\{m_1,\dots,m_{\ell-1}\}$. Since
  $a_{i_1}<a_n=m_\ell$, it follows that
  $a_{i_1}\not\in\{m_{\ell},\dots,m_{n-1}\}$ and thus
  $a_{i_1}\not\in\M$.

  For $2\le k\le n-1$ we have $a_{i_1}+\dots+a_{i_k}\ge
  a_{j_1}+a_n>m_\ell$, so
  $a_{i_1}+\dots+a_{i_k}\not\in\{m_1,\dots,m_\ell\}$. Moreover, since
  $a_{j_1}+\dots+a_{j_{k-1}}\not\in\{m_{\ell+1}-a_n,\dots,m_{n-1}-a_n\}$,
  we also have $a_{i_1}+\dots+a_{i_k}=a_{j_1}+\dots+a_{j_{k-1}}+a_n
  \not\in\{m_{\ell+1},\dots,m_{n-1}\}$.  Hence,
  $a_{i_1}+\dots+a_{i_k}\not\in\M$.

  \medskip

  \textit{Case 2: $a_n\not\in\M$.}
  Let $\M'=\{m_2-a_n,\dots,m_{n-1}-a_n\}$, and apply the induction
  hypothesis to the jumps $a_1,\dots,a_{n-1}$ and the set $\M'$. We
  get a permutation $(j_1,\dots,j_{n-1})$ of $(1,2,\dots,n-1)$ such
  that $a_{j_1}+a_{j_2}+\dots+a_{j_k}\not\in\M'$ for any $k$ with
  $1\le k\le n-2$.

  Let $\ell$, $1\le\ell\le n-1$, be the first index for which
  $a_{j_1}+\dots+a_{j_{\ell}}\ge m_1$. (By the assumption $m_1\le
  a_1+\dots+a_{n-1}$, there exists such an index.) Then choose
  $(i_1,\dots,i_n)=(j_1,\dots,j_{\ell-1},n,j_\ell,\dots,j_{n-1})$. (If
  $\ell=1$, choose $(n,j_1,\dots,j_{n-1})$.) We show that this
  permutation fulfills the desired property.

  For $1\le k\le\ell-1$ we have $a_{i_1}+\dots+a_{i_k}\le
  a_{j_1}+\dots+a_{j_{\ell-1}}<m_1$, so
  $a_{i_1}+\dots+a_{i_k}\not\in\M$.

  For $\ell\le k\le n-1$ we have $a_{i_1}+\dots+a_{i_k}\ge
  a_{j_1}+\dots+a_{j_{\ell}-1}+a_n>
  a_{j_1}+\dots+a_{j_{\ell}-1}+a_{j_\ell}\ge m_1$. Since
  $a_{j_1}+\dots+a_{j_{k-1}}\not\in\M'$, it follows that
  $a_{i_1}+\dots+a_{i_k}=a_{j_1}+\dots+a_{j_{k-1}}+a_n
  \not\in\{m_2,\dots,m_{n-1}\}$.  Hence,
  $a_{i_1}+\dots+a_{i_k}\not\in\M$.
\end{proof}


\paragraph{Acknowledgments.}
Discussions on the subject with Lajos R\'onyai are gratefully
acknowledged.  This work was supported in part by OTKA grant NK72845.



\bigskip

\noindent\textbf{G\'eza K\'os}
received his Ph.D. from Lor\'and E\"otv\"os University in Budapest,
where he currently teaches real and complex analysis. He is a member
of the editorial board of \textit{K\"oMaL}. He has served on the
problem selection committee of the International Mathematical Olympiad
in 2006, 2007, 2008, and 2010.

\noindent\textit{Department of Analysis, Lor\'and E\"otv\"os
  University, P\'azm\'any s. 1/c, Budapest, Hungary
  \\
  Computer and Automation Research Institute, Kende u. 13-17,
  Budapest, Hungary
  \\
  kosgeza@sztaki.hu}


\end{document}